\documentclass[reqno,12pt,a4paper]{amsart}
\usepackage[margin=2cm]{geometry}
\usepackage{graphicx} 
\usepackage{amsmath, amsthm, amsfonts, amssymb, amscd, bm, mathrsfs}
\usepackage{verbatim}
\usepackage{enumerate}
\usepackage[numbers]{natbib}
\usepackage[colorlinks=true,linkcolor=blue,linktoc=page]{hyperref}
\usepackage{multicol}

\usepackage{blkarray}

\usepackage{array,xcolor,colortbl}
\providecommand{\mk}{\cellcolor{red!75}}
\providecommand{\km}{\cellcolor{blue!35}}
\def\Z{\mathbb{Z}}
\newcommand{\cyc}[2]{\Z^{#1}_{#2}}
\def\gp{G_+}
\def\={\equiv}

\newtheorem{thm}{Theorem}
\newtheorem{cor}[thm]{Corollary}
\newtheorem{lemma}[thm]{Lemma}
\newtheorem{conj}{Conjecture}
\theoremstyle{definition}

\theoremstyle{remark}

\newcommand{\LA}{\mathcal{A}}   
\newcommand{\LB}{\mathcal{B}}
\newcommand{\LU}{\mathcal{U}}

\newcommand{\sym}{\mathcal S}

\renewcommand{\ge}{\geqslant}
\renewcommand{\le}{\leqslant}

\def\eref#1{$(\ref{#1})$}
\def\sref#1{\S$\ref{#1}$}
\def\lref#1{Lemma~$\ref{#1}$}
\def\tref#1{Theorem~$\ref{#1}$}

\def\cyref#1{Corollary~$\ref{#1}$}
\def\cjref#1{Conjecture~$\ref{#1}$}

\begin{document}

\title{Latin hypercubes with restricted transversals}

\author{Billy Child and Ian M. Wanless}

\thanks{
School of Mathematics, Monash University, Vic 3800, Australia.
{\tt william.child@monash.edu, ian.wanless@monash.edu}.
}

\maketitle

\begin{abstract}
  A $k$-plane of a $d$-dimensional array is a subarray formed by
  fixing $d-k$ coordinates and allowing the remaining $k$ coordinates
  to vary freely.  A \emph{Latin hypercube} of dimension $d$ and order
  $n$ is an $n\times n\times\cdots\times n$ array of dimension $d$
  containing symbols from an $n$-set, such that each $1$-plane
  contains each of the possible entries exactly once.  A
  \emph{transversal} in a Latin hypercube of order $n$ is a set of $n$
  entries of the hypercube, no pair of which agree in any coordinate
  or contain the same symbol.

  The aim of this paper is to construct Latin hypercubes that have
  transversals but which have many entries that are not in any
  transversal, or for which the number of disjoint transversals is
  limited.  We show the following results in the case when the
  dimension $d$ is even.  For all even $n\ge 10$ there exists a Latin
  hypercube of order $n$ that contains a transversal but for which all
  transversals hit one $(d-2)$-plane.  For $n\in\{6,8\}$ there exists
  a Latin hypercube of order $n$ that contains a transversal but for
  which all transversals hit one of two $(d-2)$-planes.  For even
  $d>2$ there is a Latin hypercube of order $n=4$ that contains a
  transversal but has $2^d$ entries that are not in any transversal.

  Our constructions use a quasigroup $(Q,\ast)$ to increase the
  dimension of a Latin hypercube using the rule
  $H_d(x_1,\dots,x_d)=H_{d-1}(x_1,\dots,x_{d-1})\ast x_d$.  We give
  several characterisations which allow us to diagnose which entries
  of $H_d$ are in transversals in terms of properties of $H_{d-1}$ and $Q$.
\end{abstract}

\section{Introduction}

For a positive integer $n$, let $I_n$ be an index set of cardinality
$n$.  For integer $d\ge2$, a \emph{$d$-dimensional hypercube $H$ of
  order $n$} is an array $I_n\times\cdots\times I_n\to I_n$. For
$(x_1,\dots,x_d)\in I_n\times\cdots\times I_n$, we refer to
$(x_1,\dots,x_d;H(x_1,\dots,x_d))$ as the \emph{entry} of $H$ with
\emph{coordinates} $(x_1,\dots,x_d)$ and $\emph{symbol}$
$H(x_1,\dots,x_d)$.

A \emph{submatrix} of a hypercube $H$ is a restriction of $H$ to
$J_1\times\cdots\times J_d$, where $J_i\subseteq I_n$ for each $i$. If
for all $i$ we have $|J_i|\in\{1,n\}$, then the submatrix is called a
$k$-\emph{plane}, where $k$ is the number of subsets $J_i$ with
$|J_i|=n$. We will call a $(d-1)$-plane of $H$ a \emph{hyperplane},
and a $1$-plane of $H$ a \emph{line}. A hypercube is \emph{Latin} if
each line contains every element of $I_n$.  Denote by $M(d,n)$ the set
of $d$-dimensional Latin hypercubes of order $n$.  A \emph{diagonal}
in a Latin hypercube of order $n$ is a set of $n$ entries of the
hypercube, no pair of which lie in a hyperplane. The diagonal is
\emph{constant} if all its entries contain the same symbol.  A
\emph{transversal} is a diagonal in which no pair of entries contain
the same symbol.  The study of transversals in the case when $d=2$ has
a fascinating history, including several notorious conjectures and
some impressive recent results \cite{Mon24,transurv}.  This paper is
an attempt to generalise some of the basic theory to higher dimensions.
Earlier efforts in this direction include \cite{CW20,PPV24,Tar16,Tar17,Tar18a,Tar18b,Tar20,Tar21}.

Let $\sym_n$ denote the symmetric group of degree $n$.  There is a
natural action of $(\sym_n)^{d+1}$ on $M(d,n)$. Elements of a common
orbit under this action are said to be \emph{isotopic}.  There is also
a natural action of the wreath product $\sym_n\wr\sym_{d+1}$ on $M(d,n)$.
Orbits under this action are called \emph{species}.

We will often want our index set to have a group structure.
Let $G=(I_n,+)$ be an (additive) abelian group with identity $0$.
We define $G^d\in M(d,n)$ by
$G^d(x_1,\dots,x_d)=\sum x_i$.
  Consider a hypercube $H\in M(d,n)$ indexed by $G$.
  We define a function $\Delta:H\rightarrow G$
  by $\Delta(e)= \sigma-x_1-\cdots-x_d$ 
  for each entry $e=(x_1,\dots,x_d;\sigma)$.
We let $\gp$ denote the sum of the elements of $G$. Note that $\gp=0$
unless $G$ has a unique involution, in which case $\gp$ will be that
involution.

\begin{lemma}[$\Delta$ lemma]\label{l:delta} 
Suppose that $H\in M(d,n)$ is indexed by $G$ and that
$T=\{\alpha_1,\dots,\alpha_n\}$ is a transversal of $H$. Then,
\[\sum_{i=1}^n \Delta(\alpha_i)=(1-d)\gp.\]
\end{lemma}

\begin{proof}
  The symbols in $T$ sum to $\gp$.  Also, by the definition of a
  transversal, every hyperplane of $H$ contains exactly one of the
  $\alpha_i$.  So the sum over all coordinates of the $\alpha_i$'s is
  $d\gp$. The result follows.
\end{proof}

The $\Delta$ function can be viewed as measuring the difference
between a hypercube in $M(d,n)$ and $G^{d}$; in particular, the
$\Delta$ function is zero wherever a hypercube agrees with
$G^{d}$. Hence, an immediate application of the $\Delta$-lemma is:

\begin{thm}\label{t:notrancyc}
There are no transversals in $\Z_n^d$ when $n$ and $d$ are both even.
\end{thm}

Indeed, it is known from \cite{CW20} that for even $n$ and $d$,
the number of species of hypercubes in $M(d,n)$ that have no transversals is
superexponential in $n$. The situation for odd $n$ or $d$ seems to be very
different, as captured by the following conjecture from \cite{transurv}.

\begin{conj}\label{cj:oddity}
  If $n$ is odd or $d$ is odd and $H\in M(d,n)$, then $H$ has transversals.
\end{conj}

\cjref{cj:oddity} is obvious for $n=2$ and is
known for orders $n=3,4,5$ by, respectively,
Taranenko \cite{Tar16}, Taranenko \cite{Tar17,Tar18a} and Perezhogin,
Potapov and Vladimirov \cite{PPV24}.

  A Latin hypercube is said to be a \emph{confirmed bachelor} if it has at
  least one entry that is not in any transversal.
  In \cite{MW08} the Latin hypercubes of order $n$ and dimension $d$
were enumerated for
\begin{equation}\label{e:smallnd}
  (n,d)\in\{(4,3),(4,4),(4,5),(5,3),(5,4),(5,5),(6,3)\}.
\end{equation}
Among the parameters in \eref{e:smallnd}, the only ones for which
confirmed bachelor hypercubes exist are when $(n,d)=(4,4)$.  In that
case, there are three species of confirmed bachelors. All three will
be constructed in \sref{s:ord4}.  We do not know of any confirmed
bachelor Latin hypercube when $n$ or $d$ is odd and $d>2$.  For this
reason, our strongest results in the present paper will concentrate on
the case when $n$ and $d$ are even.

  Let $L$ be a $d$-dimensional Latin hypercube indexed by $G$. For
  $d'>d$,
  we will define the \emph{$d'$-dimensional $G$-extension} of $L$, to
  be the Latin hypercube $L'$ given by
  \begin{equation}\label{e:LpL}
    L'(x_1,\dots,x_{d'})= L(x_1,\dots,x_d)+\sum_{i=d+1}^{d'}x_i.
  \end{equation}
  We also define the projection map $\pi:L'\to L$
  given by: $\pi(x_1,\dots,x_{d'};L'(x_1,\dots,x_{d'}))=(x_1,\dots,x_d;L(x_1,\dots,x_d))$.
  For a set $S$ of entries of $L$ we refer to the set $\pi^{-1}(S)$ of
  entries in $L'$ as the \emph{fibre} of $S$. Note that $L'$ is the disjoint
  union of fibres of the entries in $L$.
  For any entry $\alpha$ of $L'$ as defined by \eref{e:LpL},
  observe that $\alpha$ and $\pi(\alpha)$
  have the same $\Delta$ value (we are
  evaluating $\Delta$ functions on two different domains in this claim, one on
  $L$ and one on $L'$, but both functions map to $G$).

  For each $H\in M(d,n)$ the operation $I_n\times\cdots\times I_n\to I_n$
  which maps $(x_1,\dots,x_d)$ to $H(x_1,\dots,x_d)$ is known in
  algebraic terms as a \emph{$d$-ary quasigroup of order $n$}.  Our method of
  extending a Latin hypercube to higher dimensions is
  similar to that used by Taranenko in
  \cite{Tar17}, \cite{Tar18a}, \cite{Tar18b} and \cite{Tar21}. For a
  binary quasigroup $Q$ with operation $\ast$, define $Q[d]$ to be the
  \emph{iterated} quasigroup given by
  $Q[d](x_1,\dots,x_{d+1})=(\cdots((x_1\ast x_2)\ast x_3)\ast\cdots\ast x_d)\ast x_{d+1}$.
  Compared to Taranenko's iterated quasigroups, the dimension boosting
  strategy used in \eref{e:LpL} is restricted to the use of abelian
  groups, but it can be applied successively using a potentially
  different group each time. At the end of \sref{s:lift} we briefly
  consider a more general construction of a Latin hypercube, where an
  arbitrary binary quasigroup is used to boost each dimension.

The structure of this paper is as follows. In \sref{s:ord4} we
construct confirmed bachelor hypercubes (with transversals) of even
dimension $d>2$ and order $n\=0\bmod4$. Our results cover all possible
confirmed bachelor hypercubes in the case $d=n=4$.
In \sref{s:lift} and in \sref{s:dilation} we give two methods for
constructing larger hypercubes from smaller ones. In \sref{s:lift} we
boost the dimension while holding the order fixed and in
\sref{s:dilation} we boost the order while holding the dimension
fixed.  In both cases we show that certain properties of diagonals in
the smaller hypercube relate to diagonals in the larger hypercube.
Our results in \sref{s:lift} allow us to lift many published
restrictions on transversals in Latin squares to higher dimensional
hypercubes. We thereby construct Latin hypercubes of even dimension
and even order that have transversals but which do not have large sets
of disjoint transversals. We also show that these examples have many
entries that are not in any transversal.

\section{Hypercubes with transversals but not everywhere}\label{s:ord4}

We begin by providing some constructions for Latin hypercubes that
possess transversals, but also have some cells that are not in
transversals.

\begin{thm}\label{t:confbach}
  For all even $d>2$ and $n=0 \bmod{4}$, there is a confirmed bachelor hypercube of dimension $d$ and order $n$ that contains a transversal.
\end{thm}

\begin{proof} Start with the cyclic hypercube $\Z_n^d$ and construct a new Latin hypercube $H$ indexed by $\Z_n$ in the following way. For coordinates $(x_1,\dots, x_d)$, let $m$ be the number of $x_1,\dots,x_d$ that are odd. Then
  \[H(x_1,\dots, x_d)=
  \begin{cases}
    \sigma-(m-1) & \text{if $m$ is odd,}\\
    \sigma-m &\text{if $m$ is even,}
  \end{cases}
  \]
  where $\sigma=\cyc{d}{n}(x_1,\dots, x_d)$.
It is immediate that the $\Delta$ value of an entry in $H$ is $-(m-1)$
when $m$ is odd or $-m$ when $m$ is even, and hence is always even. We
now show that $H$ is Latin by considering a line $\ell$ that includes
$H(x_1,\dots, x_d)$ and has
free coordinate $x_i$. Let $k$ be the number of coordinates in
$(x_1,\dots,x_{i-1},x_{i+1},\dots,x_d)$ that are odd. Suppose $k$ is
even. Then the $\Delta$ value of an entry in $\ell$ is equal to $-k$
if $x_i$ is even, or $-(k+1-1)=-k$ if $x_i$ is odd. Along such a line
each symbol differs from the corresponding symbol in the cyclic
hypercube by the constant $k$, and so the line is Latin. On the other
hand, suppose $k$ is odd. Then the $\Delta$ value of an entry in
$\ell$ is equal to $-(k-1)$ if $x_i$ is even, or $-(k+1)$ if $x_i$ is
odd. Along such a line the odd symbols differ from the corresponding
symbols in the cyclic hypercube by $k-1$ and the even symbols by
$k+1$, but both are constant differences and so the line is Latin.
Since in either case $\ell$ is Latin, it follows that $H$ is Latin.

Let $I$ be the submatrix of $H$ consisting of the entries whose
coordinates are all in $\{0,1\}$. For an entry of $I$, if the number
of odd coordinates is odd, then the symbol is $1$, and if the number
of odd coordinates is even, the symbol is $0$. In other words, $I$
equals $\cyc{d}{2}$. Let $H'$ be the Latin hypercube obtained from $H$
by switching every 0 with a 1 and vice versa inside $I$ to make a
subhypercube $I'$. The $\Delta$ values in $H'$ remain even outside
$I'$, but are odd for every entry in $I'$.

We next argue that there is no transversal $T$ of $H'$ that hits $I'$.
By \lref{l:delta}, the sum of the $\Delta$ values of the entries
in $T$ equals $n/2 \bmod{n}$, which is even since $n=0 \bmod{4}$. As
the $\Delta$ values are odd on $I'$ and even elsewhere, $T$ must
contain an even number of entries from $I'$. Suppose that $T$ hits an
entry $\alpha$ in $I'$. There is a unique entry $\alpha'$ in $I'$ that
does not share any coordinate with $\alpha$. However since $d$ is
even, $\alpha'$ will have the same symbol as $\alpha$. Thus there is
no other entry in $I'$ that can be in $T$ together with $\alpha$. Hence
$T$ cannot exist.

It remains to show that there is a transversal in $H'$. For $i=0,2,4,\dots,\tfrac{n}{2}-2$, each row in the table below gives an entry, $(x_1,\dots,x_d;\sigma)$, of a transversal.
\begin{equation}\label{e:atran}
\begin{array}{c c c c c c |c c c}
  x_1&  x_2&  x_3 & x_4 & \cdots& \cdots& \sigma &\Delta\\ \hline
  3+i &1+i& 1+i &1-i&  i& -i&  2+2i& -4\\
  \tfrac{n}{2}+3+i&\tfrac{n}{2}+1+i&-1-i&2+i&1+i&1-i &3+2i&2-d\\
  2+i&-i&i&3+i&\tfrac{n}{2}+1+i&\tfrac{n}{2}+1-i&5+2i&4-d\\
  -i&2+i&\tfrac{n}{2}+i&\tfrac{n}{2}+2+i&\tfrac{n}{2}+i&\tfrac{n}{2}-i&4+2i&0\\[-1ex]
  &&&&\multicolumn{2}{c}{\underbrace{\rule{35mm}{0pt}}}&&\\
  &&&&\multicolumn{2}{c}{\tfrac{d-4}{2} \text{ times}}
\end{array}
\end{equation}

None of the entries in \eref{e:atran} is in $I'$, since for each entry
some pair within $\{x_1,x_2,x_3,x_4\}$ differs by $2$. So, the values
of $\sigma$ and $\Delta$ can be easily calculated with the rules for
the construction above. It can be checked that each coordinate column
and the symbol column take all values from $0,\dots,n-1$, and hence we
have a transversal.
\end{proof}

It is not hard to show (or compute) that when $n=d=4$ the hypercube
constructed in \tref{t:confbach} has transversals through every entry
not in $I'$, and hence has precisely 16 entries that are not
in any transversal.

As mentioned after \eref{e:smallnd}, there are three confirmed
bachelor hypercubes when $(n,d)=(4,4)$.  These include $\cyc{4}{4}$
(which has no transversals by \tref{t:notrancyc}) and the hypercube
constructed in the proof of \tref{t:confbach}. For completeness we now
construct a representative of the only other species of confirmed
bachelors.

 In this construction, the index set is $\Z_4$.
 Let $H$ be the hypercube obtained by adding 2 to the
 symbol in every cell $(i,j,k,l)$ of $\cyc{4}{4}$ for which
 $i+j\=k\=l\bmod 2$.  Note that $H$ is Latin because whenever we added 2 to
 the entry in $(i,j,k,l)$ we also added 2 to the entries in
 cells $(i+2,j,k,l)$, $(i,j+2,k,l)$, $(i,j,k+2,l)$ and $(i,j,k,l+2)$.
 Next we create a Latin hypercube $H'$ by switching the symbols 0 and 3
 in entries $(i,j,k,l)$ of $H$ for which $\{k,l\}\subseteq\{2,3\}$, creating
 the following entries
\begin{gather}\label{e:swit03}
 \begin{aligned}
   (i,3-i,k,l;0),
   (i,2-i,k,l;3)&\qquad\text{for $i\in\Z_4$ and $k=l\in\{2,3\}$},\\
   (i,3-i,k,l;3),
   (i,2-i,k,l;0)&\qquad\text{for $i\in\Z_4$ and $\{k,l\}=\{2,3\}$}.
 \end{aligned}
 \end{gather}
We claim that none of the $32$ entries in \eref{e:swit03} is in a
transversal.  Suppose there is a transversal $T$ that hits one of
these entries. Note that the entries in \eref{e:swit03} are the only
entries with odd $\Delta$ value in $H'$, and \lref{l:delta} says
that the sum of the $\Delta$-values over $T$ must equal $2$ (in
particular, it is even). So we conclude that $T$ must hit at least two
entries among those in \eref{e:swit03}, which can only be one with
symbol 0 and one with symbol 3. Checking the different options reveals
that these two entries, which must disagree in their values of $k$ and
$l$, must have combined $\Delta$ value equal to $0\bmod4$.

The other two entries in $T$ must have the symbols 1 and 2, chosen from among
\begin{gather}\label{e:unswit12}
 \begin{aligned}
   (i,1-i,k,l;1),
   (i,0-i,k,l;2)&\qquad\text{for $i\in\Z_4$ and $k=l\in\{0,1\}$},\\
   (i,1-i,k,l;2),
   (i,0-i,k,l;1)&\qquad\text{for $i\in\Z_4$ and $\{k,l\}=\{0,1\}$}.
 \end{aligned}
 \end{gather}
Again, any choice of a 1 and a 2 from \eref{e:unswit12} have combined
$\Delta$ value equal to $0\bmod4$. So it is impossible to satisfy the
requirement that the sum of the $\Delta$ values across $T$ is $2$.
This completes the proof that the 32 entries in \eref{e:swit03}
are not in any transversal. A computation shows that the remaining
224 entries in $H'$ are in transversals. One example of a transversal is
\[
\{(3,1,2,0;0),\ (0,2,0,3;1),\ (2,0,3,1;2)\ (1,3,1,2;3)\}.
\]

We have constructed representatives of each of the three species of
confirmed bachelors for $(n,d)=(4,4)$; the other 23 species have
transversals through every cell. Of these 23 species, 17 have a
decomposition into transversals. Such a decomposition consists of 64
disjoint transversals.  To show there is no such decomposition in the
other 6 species, it suffices to exhibit a set of fewer than 64 entries
such that every transversal hits that set. The 6 species can be
constructed from $\Z_4^4$ by ``turning'' one, two or three disjoint
subhypercubes of order 2.  Here, by turning a subhypercube of order 2,
we mean replacing it by the other possible hypercube of order 2 on the
same two symbols. Each hypercube of order 2 and dimension 4 contains
16 entries. So if we turn three or fewer subhypercubes of order 2 in
$\Z_4^4$ then there cannot be a decomposition into transversals. There
are one, two and three species, respectively, that can be formed from
$\Z_4^4$ by turning one, two and three disjoint subhypercubes of order
2.

For the other catalogues reported in \eref{e:smallnd}, we attempted to
use a random hill-climbing algorithm to find decompositions into
transversals for each species. Our algorithm was not powerful enough
to cope with dimension 5, but was successful in all other cases. In
other words, we confirmed that whenever
$(n,d)\in\{(4,3),(5,3),(6,3),(5,4)\}$ every Latin hypercube has a
decomposition into transversals.

Building on the idea of turning subhypercubes, we can show:

\begin{thm}\label{t:2tod}
  For all even $d$ and even $n>2$ there exists a hypercube in $M(d,n)$
  which has transversals but for which no set of disjoint transversals
  has cardinality exceeding $2^{d}$.
\end{thm}

\begin{proof}
  Start with $\Z_n^d$ and replace a subhypercube of order $2$ by the
  other possible subhypercube on the same symbols. Specifically, let
  $I=\{0,n/2\}^d$ be the set of cells for which every coordinate is in
  $\{0,n/2\}$. Now add $n/2$ to the symbol in each cell in $I$ to form
  $H\in M(d,n)$. By \lref{l:delta}, every transversal of $H$ must
  include an entry that differs from $\Z_n^d$.  Since $|I|=2^d$, there
  are not more than $2^d$ disjoint transversals in $H$.

  It remains to show that $H$ has a transversal, consisting of the
  following entries $(x_1,\dots,x_d;\sigma)$:
\begin{equation}\label{e:btran}
  \begin{array}{c c c c |c c c c}
    x_1&x_2&  \cdots& \cdots& \sigma\\
    \cline{1-5}
    0&  0& 0 & 0 & \tfrac{n}{2}\\
    i&  -2i& i & -i & -i && \text{for $1\le i<\tfrac{n}{2}$,}\\
    i&-1-2i& i& -i & -1-i && \text{for $\tfrac{n}{2}\le i<n$.}\\[-1ex]
  &&\multicolumn{2}{c}{\underbrace{\rule{17mm}{0pt}}}&&\\
  &&\multicolumn{2}{c}{\tfrac{d-2}{2} \text{ times}}
\end{array}\qedhere
\end{equation}
\end{proof}

With a little effort it is possible to show that each $H$ constructed
in \tref{t:2tod} contains $2^d$ disjoint transversals obtained by
translating the cells of the transversal in \eref{e:btran} by some
vector in $\{0,\tfrac{n}{2}\}^d$.  Based on computations for small
values of the parameters $(n,d)$ we conjecture that every entry of each
$H$ will be in a transversal.

\section{Extending to higher dimensions}\label{s:lift}

In this section we consider the relationship between hypercubes of
different dimensions that are related by $G$-extension, or by a more
general extension involving quasigroups.  In particular, we are
interested in how properties of diagonals (including transversals) in
the lower dimensional hypercube lift to diagonals of the extension.

A key tool for analysing our extensions will be the following theorem
from Marshall Hall \cite{HallJrTheorem}.

\begin{thm}\label{MarshallHall}
Given a sequence $(\sigma_1,\dots,\sigma_n)$, of $n$ (not
necessarily distinct) elements of a finite abelian group $G$ of
order $n$, such that $\sum_{i=1}^n\sigma_i=0$, there exist
permutations $(a_i)$ and $(b_i)$ such that $a_i-b_i=\sigma_i$ for
$i=1,\dots, n$.
\end{thm}

Let $L$ be a $d$-dimensional Latin hypercube indexed by $G$. Then a
diagonal $D$ of $L$ is \emph{$(G,d')$-suitable} if the sum of the
$\Delta$-values of the entries of $D$ is $(1-d')\gp$.

\begin{thm}\label{t:boostd}
Let $L$ be a $d$-dimensional Latin hypercube of order $n$ and $L'$
be a $d'$-dimensional $G$-extension of $L$ for some $d'> d$.  Then
any entry $\alpha$ of $L'$ is contained in a transversal if and
only if $\pi(\alpha)$ is contained in a $(G,d')$-suitable diagonal in $L$.
\end{thm}

\begin{proof}
The forward implication follows immediately from the projection map
preserving the $\Delta$ value. That is, suppose
$T=(\alpha_1,\dots,\alpha_n)$ is a transversal of $L'$. Then
$D=(\beta_1,\dots,\beta_n)$ given by $\beta_i=\pi(\alpha_i)$ is a
diagonal of $L$ with $\Delta$ values that sum to the required value.

To show the converse, suppose $D=(\beta_1,\dots,\beta_n)$ is a
$(G,d')$-suitable diagonal in $L$, where
$\beta_i=(x^i_1,\dots,x^i_d;\sigma_i)$. 
Our first goal will be to find a transversal of $L'$ that $\pi$
projects to $D$.

We start by considering the case where $d'-d=1$.
Since $D$ is a $(G,d')$-suitable diagonal,
\begin{align*}
  (1-d')\gp=\sum_{i=1}^n \Delta(\beta_i)
  &= \sum_{i=1}^n \sigma_i - \sum_{j=1}^d\sum_{i=1}^n  x_j^i
  =\sum_{i=1}^n \sigma_i - \sum_{j=1}^d\gp=\sum_{i=1}^n \sigma_i - d\gp.
\end{align*}
Hence $\sum_i\sigma_i=0$ and we can apply \tref{MarshallHall},
giving permutations $(a^i)$ and $(b^i)$ of $\{ 1,\dots,n \}$ such that $\sigma_i = b^i - a^i$. Let $D'=(\beta'_1,\dots,\beta'_n)$ where $\beta'_i = (x^i_1,\dots,x^i_d,a^i;b^i)$. Then $D'$ has the same $\Delta$ sum as $D$ and contains the symbols $1,\dots,n$, that is, $D'$ is a transversal.

Next, suppose $d'>d+1$ and let $L^*$ be the $G$-extension of $L$ to
$d+1$ dimensions. Then, let $D^*=(\beta^*_1,\dots,\beta^*_n)$ be the
diagonal in $L^*$ with
$\beta^*_i=(x^i_1,\dots,x^i_d,a^i;\sigma_i+a^i)$ for some arbitrary
permutation $(a^i)$ of $1,\dots,n$. Note that $D^*$ is a
$(G,d')$-suitable diagonal in $L^*$ that projects onto $D$. By
repeated application of this process we reduce this case to the
$d'=d+1$ case above. In all cases then, we have a transversal $D'$
that projects onto $D$.

Let $\alpha=(y_1,y_2,\dots,y_{d'};\tau)$ be any entry of $L'$
such that $\pi(\alpha)=\beta_k$ for some $k$.  Let
$\beta'_k=(x_1^k,\dots,x_{d'}^k;\sigma'_k)$ be the entry of $D'$
that satisfies $\pi(\beta'_k)=\beta_k$. Define
$V=(y_1-x_1^k,\dots,y_{d'}-x_{d'}^k)$ and note that the first $d$
coordinates of $V$ are zero.  Consider the $n$-tuple $T$ of entries
of $L'$ located by shifting each entry of $D'$ by adding $V$ to
its coordinates. The $k$-th entry of $T$ is $\alpha$, by
construction. It remains to argue that $T$ is a transversal of
$L'$.  In each coordinate the elements of $T$ have been translated
by a constant relative to the elements of $D'$, which are pairwise
different in each coordinate. Also, the symbols in $T$ have each been
shifted by the sum of the elements of $V$, relative to the symbols
in $D'$. As the symbols in $D'$ are all distinct, the same is true
of the symbols in $T$.
\end{proof}

\begin{cor}\label{cy:TransRes}
  Let $P_1,\dots,P_m$ be $k$-planes in a $d$-dimensional Latin hypercube $L$.
  Suppose that all
  $(G,d')$-suitable diagonals of $L$ pass through $P_1\cup\cdots\cup P_m$.
  Then all transversals in the $d'$-dimensional 
  $G$-extension of $L$ pass through $P'_1\cup\cdots\cup P'_m$, where
  $P'_i=\pi^{-1}(P_i)$ is a $(k+d'-d)$-plane, for $1\le i\le m$.
\end{cor}

Examining the proof of \tref{t:boostd} we have:

\begin{cor}\label{c:mnd}
Let $L$ be a $d$-dimensional Latin hypercube of order $n$ and $L'$
be a $d'$-dimensional $G$-extension of $L$ for some $d'> d$.
If $L$ has $m$ disjoint $(G,d')$-suitable diagonals then $L'$ has
$mn^{d'-d}$ disjoint transversals.
\end{cor}

In particular, if $m=n^d$ we have:

\begin{cor}\label{c:decomp}
Let $L$ be a $d$-dimensional Latin hypercube of order $n$ and $L'$
be a $d'$-dimensional $G$-extension of $L$ for some $d'>d$.
If $L$ has a decomposition into $(G,d')$-suitable diagonals then $L'$ has
a decomposition into disjoint transversals.
\end{cor}

Consider a diagonal in a $d$-dimensional Latin hypercube (indexed by
some abelian group of order~$n$) on which every symbol is
$\sigma$. The sum of the $\Delta$ values along this diagonal will be
$n\sigma-d\gp=-d\gp$. In particular, for Latin squares the sum of
$\Delta$ values along any constant diagonal is $0$, meaning the
diagonal is $(G,d')$-suitable unless $d'$ is even and $G$ has a unique
involution. Also, every Latin square trivially decomposes into
constant diagonals. So applying \cyref{c:decomp}, we have:

\begin{cor}\label{cy:decomt}
Let $L$ be a Latin square of order $n$ and $L'$
be a $d'$-dimensional $G$-extension of $L$ for some $d'>2$.
If $d'$ is odd or $n$ is odd or $G$ has a noncyclic Sylow $2$-subgroup
then $L'$ has a decomposition into transversals.
\end{cor}

Note that some Latin hypercubes do not have any constant diagonals,
let alone a decomposition into constant diagonals. One family of such
examples is the cyclic Latin hypercubes $\Z_n^d$ of even order $n$ and odd
dimension $d$. In such hypercubes, the sum of the $\Delta$ values along
any constant diagonal is $-dn/2\equiv n/2\bmod n$, making the diagonal
$(\Z_n,d+1)$-suitable. Hence, if there was a constant diagonal then
\tref{t:boostd} would imply the existence of a transversal in the
cyclic Latin hypercube of order $n$ and dimension $d+1$, contradicting
\tref{t:notrancyc}.

It is important to note that permuting symbols can change whether
there is $(G,d')$-suitable diagonal. Consider this Latin square, which
is isotopic to the Cayley table of $\Z_6$ via a permutation of the
symbols:
\[\left[
\begin{array}{cccccc}
  0&1&2&3&5&\mk4\\
  1&2&3&5&\mk4&0\\
  2&3&5&\mk4&0&1\\
  \mk3&5&4&0&1&2\\
  5&4&\mk0&1&2&3\\
  4&\mk0&1&2&3&5\\
\end{array}
\right].
\]
The marked diagonal has sum $3\bmod6$. By \tref{t:boostd}, if we
cyclically develop this square to any even dimension $d'\ge4$, the
resulting hypercube will have a transversal. This means the hypercube
is not isotopic to the cyclic hypercube, even though it was developed
from a Latin square that is isotopic to the cyclic Latin square.

The $\Delta$-Lemma has been a highly successful tool for showing the
absence of transversals through specific cells in Latin squares.  In
certain circumstances, \tref{t:boostd} allows us to translate such
restrictions to higher dimensional extensions of the Latin
squares. However, it is important to recognise that not all arguments
about restrictions on transversals can be translated to higher
dimensions. Some arguments in the literature boil down to a discussion
of suitable diagonals, while others rely on additional considerations.
As a concrete example, the $\Delta$-Lemma is used in \cite{WW06} to
show that certain cells in Latin squares of odd order are not in
transversals.  However, \cyref{cy:decomt} says that any $G$-extension
of these squares to dimension $d'>2$ will have a decomposition into
transversals.  The argument in \cite{WW06} relies on certain pairs of
cells containing the same symbol and hence being incompatible in a
transversal. When extending to higher dimensions such an argument
breaks because we may choose different symbols from the fibres of the two
cells.

For our next result we make use of the following theorem from \cite{GW25}.

\begin{thm}\label{t:shared}
  For every even $n\ge10$ there exists a Latin square of order $n$
  which has transversals, but all $(\Z_n,2)$-suitable diagonals
  coincide on $\lfloor n/6\rfloor$ entries.
\end{thm}

We will also need constructions for $n<10$.
For $n=8$, consider the following Latin square, with two disjoint transversals
highlighted and the $\Delta$ values shown on the right: 
\begin{equation}\label{e:ord8}
  \left[
  \begin{array}{cccccccc}
\mk0&\km1&3&4&5&6&7&2\\
3&4&\mk2&6&7&\km5&1&0\\
\km2&3&4&5&6&7&0&\mk1\\
1&2&5&3&\mk4&0&6&\km7\\
4&5&\km6&\mk7&0&1&2&3\\
5&\mk6&7&\km0&1&2&3&4\\
6&7&0&1&2&\mk3&\km4&5\\
7&0&1&2&\km3&4&\mk5&6\end{array}
\right]
\qquad
\left[\begin{array}{cccccccc}
0&0&1&1&1&1&1&3\\
2&2&\mk-1&2&2&\mk-1&2&0\\
0&0&0&0&0&0&0&0\\
-2&-2&0&-3&-3&0&-3&-3\\
0&0&0&0&0&0&0&0\\
0&0&0&0&0&0&0&0\\
0&0&0&0&0&0&0&0\\
0&0&0&0&0&0&0&0\end{array}\right]
\end{equation}

All diagonals with $\Delta$ sum equal to $4$ pass through one of the coloured
entries in the second row with $\Delta$ value $-1$. It follows that
there are at most two disjoint transversals. 

For $n = 6m$ consider the following construction. Starting with
$\cyc{2}{6m}$, we remove the entries
\begin{align*}
&(0,0,0),\ (0,m,m),\ (m,0,m),\ (m,m,2m),\\
&(2m,0,2m),\ (0,2m,2m),\ (2m,2m,4m),\
(0,4m,4m),\ (2m,4m,0),
\end{align*}
and replace them with
\begin{align*}
  &(0,0,m),\ (0,m,2m),\ (m,0,2m),\ (m,m,m),\\
  &(2m,0,0),\ (0,2m,4m),\ (2m,2m,2m),\ (0,4m,0),\ (2m,4m,4m).
\end{align*}
The resulting Latin square has
the following nonzero $\Delta$ values:
\begin{equation}\label{e:ord6}
  \begin{array}{c|cccc}
    &0&m&2m&4m\\
    \hline
    0&m&m&2m&2m\\
    m&\mk m&\km -m\\	
    2m&-2m&&-2m&-2m
  \end{array}
\end{equation}
Any diagonal with $\Delta$ sum $3m$ must pass through one of the
coloured entries with nonzero $\Delta$ value in row $m$. Hence there are at
most two disjoint transversals.

These constructions along with \cyref{cy:TransRes} give the following.

\begin{thm}\label{t:crampedT}
Suppose that $d$ is even. For all even $n\ge 10$ there exists a
hypercube in $M(d,n)$ that contains a transversal but for which all
transversals hit one $(d-2)$-plane.  For $n\in\{6,8\}$ there exists a
hypercube in $M(d,n)$ that contains a transversal
but for which all transversals hit one of two $(d-2)$-planes.
\end{thm}

\begin{proof}
We will start with a Latin square $L$ of order $n$ and take its $\Z_n$
extension to dimension $d$. We will then apply \cyref{cy:TransRes},
with $k=0$. Since $d$ is even, the $(\Z_n,d)$-suitable
diagonals of $L$ are precisely the $(\Z_n,2)$-suitable diagonals.

If $n\ge10$ then we set $L$ equal to the Latin square in \tref{t:shared}.
If $n=8$ then we set $L$ equal to the Latin square in \eref{e:ord8}.
If $n=6$ then we use $L$ from \eref{e:ord6}, which gives us the following Latin square with two disjoint transversals highlighted.
\def\ast{\rlap{$^*$}}
  \[
  \left[
\begin{array}{cccccc}
  1&2&\mk4&\km3&0&5\\
  \mk2\ast&\km1\ast&3&4&5&0\\
  \km0&3&2&5&4&\mk1\\
  3&4&\km5&\mk0&1&2\\
  4&\mk5&0&1&\km2&3\\
  5&0&1&2&\mk3&\km4\\
\end{array}\right]
  \]
These transversals are $(\Z_n,d)$-suitable diagonals and by inspection
of \eref{e:ord6}, all $(\Z_n,d)$-suitable diagonals hit one of the two
entries marked with asterisks.
\end{proof}

The examples built in \tref{t:crampedT} have many entries not in transversals.

\begin{thm}\label{t:transfree}
Suppose that $d$ is even. For all even $n\ge 10$ there exists a
hypercube in $M(d,n)$ that has transversals but has at least
$2n^{d-1}-2n^{d-2}$ entries that are not in any transversal.  For
$n\in\{6,8\}$ there exists a hypercube in $M(d,n)$ that has
transversals but contains at least $n^{d-1}-2n^{d-2}$ entries that are
not in any transversal.  If $d>2$ then there exists a hypercube in $M(d,4)$
that has transversals but contains at least $2^d$ entries that are not
in any transversal.
\end{thm}

\begin{proof}
  For $n=4$ the result follows from the proof of \tref{t:confbach}.
  For larger $n$, we use the hypercube from \tref{t:crampedT}.  First
  suppose that $n\ge10$, so there is a $(d-2)$-plane $P$ which every
  transversal hits. Now $P$ is the intersection of two hyperplanes and
  the $2n^{d-1}-2n^{d-2}$ entries in those hyperplanes but outside $P$
  cannot be in any transversal.

  The case when $n\in\{6,8\}$ is similar except now we have two
  $(d-2)$-planes $P_1$ and $P_2$ that lie in a common hyperplane.
  The $n^{d-1}-2n^{d-2}$ entries in the hyperplane but outside of
  $P_1\cup P_2$ are not in any transversal.
\end{proof}

For $n\in\{6,8\}$, every transversal hits exactly one of the two
$(d-2)$-planes mentioned in \tref{t:crampedT}, because these two
planes lie within a common hyperplane.  All Latin squares of order 6
or 8 with transversals have at least two disjoint transversals
\cite{EW12}, so \tref{t:crampedT} cannot be improved
using other Latin squares as base cases for these orders. It is not
out of the question that there are Latin hypercubes of these orders
but of dimension $>2$ that have transversals but all transversals
hit a single $(d-2)$-plane.

From \cyref{c:mnd} we then have:

\begin{cor}\label{cy:fewdisj}
  Suppose that $d$ is even. For all even $n\ge 10$ there exists a
  hypercube in $M(d,n)$ for which the largest set of disjoint
  transversals has cardinality $n^{d-2}$. For $n\in\{6,8\}$ there
  exists a hypercube in $M(d,n)$ for which the largest set of
  disjoint transversals has cardinality $2n^{d-2}$.
\end{cor}

\begin{proof}
  The upper bound comes from \tref{t:crampedT}, and the lower bound comes from
  \cyref{c:mnd}.
\end{proof}

We have relied heavily on a construction from \cite{GW25} to prove
\tref{t:crampedT} and \cyref{cy:fewdisj}. However, it is worth
remarking that a result nearly as strong can be proved using the three
families $\LA_n$, $\LB_n$ and $\LU_n$ constructed in \cite{EW12}. It
is also relevant to observe that \tref{t:2tod} gave a tighter upper
bound than \cyref{cy:fewdisj} on the number of disjoint transversals,
although it did so without showing that there are cells not in
transversals.

We end the section by showing a limitation to the strategy of
extending using a binary quasigroup in order to try to create Latin
hypercubes with restricted transversals.

\begin{thm}
Let $Q=(I_n,\ast)$ be a binary quasigroup of order $n$. Let
$H_d$ and $H_{d-1}$ be Latin hypercubes of dimensions $d\ge2$ 
and $d-1$ respectively related by
$H_d(x_1,\dots,x_d)=H_{d-1}(x_1,\dots,x_{d-1})\ast x_d$. 
\begin{enumerate}
  \item
If $H_{d-1}$ is covered by (respectively decomposes into) constant diagonals
then $H_d$  is covered by (respectively decomposes into) transversals.
  \item
If $H_{d-1}$ is covered by (respectively decomposes into) transversals
then $H_d$  is covered by (respectively decomposes into) constant diagonals.
\end{enumerate}
\end{thm}

\begin{proof}
Let $D=(\beta_1,\dots,\beta_n)$ be a constant diagonal in $H_{d-1}$,
where $\beta_i=(x_1^i,\dots,x_{d-1}^i;\sigma)$. Let
$\tau=(\tau_1,\dots,\tau_n)$ be any permutation of $I_n$.  For each
$i$, define $\beta_i'=(x_1^i,\dots,x_{d-1}^i,\tau_i;\sigma_i')$ where
$\sigma_i'=\sigma\ast \tau_i$.
Since $Q$ is a quasigroup, $\{\sigma_i':i\in I_n\}=Q$ and
$T=(\beta_1',\dots,\beta_n')$ is a transversal of $H_d$.  If instead
of $\tau$ we use another permutation $\tau'$ of $I_n$ which is
discordant with $\tau$ then the resulting transversal is disjoint from
$T$. Hence by using $n$ mutually discordant permutations of $I_n$ we can
decompose the fibre of $D$ into transversals. Part (1) of the theorem follows.

Let $T=(\beta_1,\dots,\beta_n)$ be a transversal in $H_{d-1}$, where $\beta_i=(x_1^i,\dots,x_{d-1}^i;\sigma_i)$. Let $\sigma'\in I_n$. Since $Q$ is a quasigroup, there are $\tau_i$, $i\in I_n$, such that $\sigma_i\ast \tau_i=\sigma'$. Let $D=(\beta_1',\dots, \beta_n')$, where $\beta_i'=(x_1^i,\dots,x_{d-1}^i,\tau_i;\sigma')$. Then $D$ is a constant diagonal in $H_d$. Since the choice of symbol $\sigma'$ was arbitrary, and different choices give disjoint diagonals in $H_d$, we get $n$ mutually disjoint constant diagonals in the fibre of $T$. Part (2) follows.
\end{proof}

Noting that every Latin square permits a decomposition into
constant diagonals, we conclude the following by induction:

\begin{cor}\label{cy:compred}
  For $d\ge 2$, suppose that $H\in M(d,n)$ can be defined by
  \[
  H(x_1,\dots,x_{d})=(\cdots((x_1\ast_1 x_2)\ast_2 x_3)\ast_3\cdots)\ast_{d-1}x_d
  \]
  where $\ast_1,\ast_2,\dots,\ast_{d-1}$ are (potentially different)
  binary quasigroup operations.  Then $H$ decomposes into constant
  diagonals if $d$ is even, or $H$ decomposes into transversals if $d$
  is odd.
\end{cor}

The hypercubes in \cyref{cy:compred} are a special case of what
Taranenko \cite{Tar16,Tar18a} calls completely reducible quasigroups.

\section{Dilations}\label{s:dilation}

In this section we examine a natural way to increase the order of a Latin
hypercube in a way that can preserve certain restrictions on suitable
diagonals.  For a Latin hypercube $H\in M(d,n)$, indexed by $\Z_n$, we
will call the $\lambda$-\emph{dilation} of $H$ the Latin hypercube
$H'$ indexed by $\Z_{\lambda n}$ defined by
\[H'(i_1,\dots,i_d)=
\begin{cases}
  \lambda\,H({i_1}{\lambda^{-1}},\dots,{i_d}{\lambda^{-1}}) & \text{if $i_1\equiv\cdots\equiv i_d\equiv0\bmod \lambda$},\\
  i_1+\cdots+i_d & \text{otherwise.}
\end{cases}
\]
Intuitively, $H'$ is created from $\cyc{d}{\lambda n}$ by replacing a
subhypercube isotopic to $\cyc{d}{n}$ by a subhypercube isotopic to
$H$.  The two subhypercubes contain the same symbols, which justifies
the claim that $H'$ is Latin. There are other ways to embed copies of
smaller hypercubes in larger ones. For example, there are copies of
$\Z_n^d$ in $(\Z_n\times\Z_2)^d$. However, we will not consider these
more general embeddings here.

Note that $H'$ has the following $\Delta$ values:
\[\Delta(H'(i_1,\dots,i_d))=
\begin{cases}
  \lambda\,\Delta(H({i_1}{\lambda^{-1}},\dots,{i_d}{\lambda^{-1}})) & \text{if $i_1\equiv\cdots\equiv i_d\equiv0\bmod \lambda$},\\
  0 & \text{otherwise.}
\end{cases}
\]
Define $\Psi:H\rightarrow H'$ by $\Psi(i_1,\dots,i_d;\sigma)=(\lambda i_1,\dots,\lambda i_d;\lambda\sigma)$. If $e$ is any entry in $H$
then $\lambda\Delta(e)\=\Delta(\Psi(e))$,
(the two $\Delta$'s in this equation are different functions).

\begin{lemma}\label{l:dil}
  Suppose that $d,n,\lambda$ are integers, each at least $2$, such
  that $n$ is even, $d$ is odd or $\lambda$ is odd.  Let $\alpha$ be
  an entry of a Latin hypercube $H\in M(d,n)$ indexed by $\Z_n$. Let
  $\alpha'=\Psi(\alpha)$ be the corresponding entry of $H'$, the
  $\lambda$-dilation of $H$. Let $X=\Delta^{-1}(\Z_n\setminus\{0\})$
  be the set of entries of $H$ with nonzero $\Delta$ value, and suppose
  that $\alpha\in X$.
  \begin{itemize}
  \item[(i)] If $\alpha$ is contained in a $(\Z_n,d)$-suitable
    diagonal of $H$ then $\alpha'$ is contained in a
    $(\Z_{\lambda n},d)$-suitable diagonal of $H'$.
  \item[(ii)] Suppose that any partial diagonal $D\subseteq X$ in
    $H$ containing $\alpha$ extends to a diagonal $E$ of $H$ that
    satisfies $E\cap X=D$. Then $\alpha$ is contained in a $(\Z_n,d)$-suitable
    diagonal of $H$ if and only if $\alpha'$ is contained in a
    $(\Z_{\lambda n},d)$-suitable diagonal of $H'$.
  \end{itemize}
\end{lemma}

\begin{proof}
  The parity conditions on $\lambda,d,n$ ensure that
  $\lambda(1-d)\Z_n^+\=(1-d)\Z_{\lambda n}^+\bmod\lambda n$, meaning
  that the $\Delta$ sum for a $(\Z_{\lambda n},d)$-suitable diagonal
  in $H'$ is $\lambda$ times the $\Delta$ sum for a
  $(\Z_n,d)$-suitable diagonal in $H$.
  
  Suppose that $D$ is a $(\Z_n,d)$-suitable diagonal of $H$ containing
  $\alpha$. Let $D'$ be the partial diagonal of $H'$ that is the image
  of $D$ under $\Psi$. Since $D$ is a $(\Z_n,d)$-suitable diagonal, it has
  $\Delta$ sum equal to $(1-d)\Z_n^+$, so the $\Delta$ sum over $D'$
  is $\lambda(1-d)\Z_n^+$.  Moreover, if we extend $D'$ to a diagonal
  of $H'$ in an arbitrary way, we will not change the $\Delta$ sum. This
  is because the only entries in $H'$ with nonzero $\Delta$ value are in
  the image of $\Psi$, and all such entries lie in a
  hyperplane that contains one of the entries in $D'$. Part (i) follows.

  Suppose the hypothesis of part (ii) and that $D'$ is a
  $(\Z_{\lambda n},d)$-suitable diagonal of $H'$ that contains $\alpha'$. Define
  $D=X\cap\{e\in H:\Psi(e)\in D'\}$. Note that $D$ is a partial
  diagonal of $H$ containing $\alpha$. By assumption, $\alpha$ is
  contained in a diagonal of $H$ with the same $\Delta$ sum as
  $D$. This diagonal will be $(\Z_n,d)$-suitable, given the properties
  of $\Psi$ and assumptions on the parities of $d,n,\lambda$.
\end{proof}

Consider the following Latin square $L_8$, indexed by $\Z_8$:
\[L_8=\left[
\begin{array}{cccccccc}
0& 1& 2& 3& 4& 5& 6& 7\\
1& 4& 5& 6& 7& 0& 3& 2\\
2& 3& 4& 5& 6& 7& 0& 1\\
3& 6& 7& 0& 1& 2& 5& 4\\
4& 5& 6& 7& 0& 1& 2& 3\\
5& 0& 1& 2& 3& 4& 7& 6\\
6& 7& 0& 1& 2& 3& 4& 5\\
7& 2& 3& 4& 5& 6& 1& 0\\
\end{array}
\right]
\qquad
\left[
\begin{array}{cccccccc}
0& 0& 0& 0& 0& 0& 0& 0\\
0& \mk2& \mk2& \mk2& \mk2& \mk2& \mk4& \mk2\\
0& 0& 0& 0& 0& 0& 0& 0\\
0& \mk2& \mk2& \mk2& \mk2& \mk2& \mk4& \mk2\\
0& 0& 0& 0& 0& 0& 0& 0\\
0& \mk2& \mk2& \mk2& \mk2& \mk2& \mk4& \mk2\\
0& 0& 0& 0& 0& 0& 0& 0\\
0& \mk2& \mk2& \mk2& \mk2& \mk2& \mk4& \mk2\\
\end{array}
\right]
\]
The matrix on the right shows the $\Delta$ values for $L_8$ with
shading highlighting the positions of the entries in $X$.  Given these
$\Delta$ values, it is easy to argue that $L_8$ has no
$(\Z_8,2)$-suitable diagonal, and hence has no transversal. However,
the $2$-dilation of $L_8$ (which is indexed by $\Z_{16}$) does have
$(\Z_{16},2)$-suitable diagonals; indeed it has transversals through
every entry. This example shows that the hypothesis in
\lref{l:dil}(ii) cannot be omitted. However, it is not necessarily an
easy condition to verify.  So we next give a condition that is much
easier to check and which is strong enough to lift many arguments that
use the $\Delta$-Lemma from $H$ to $H'$.

\begin{lemma}\label{l:dilrect}
  Suppose that $d,n,\lambda$ are integers, each at least $2$, such
  that $n$ is even, $d$ is odd or $\lambda$ is odd.  Suppose that
  $H\in M(d,n)$ is indexed by $\Z_n$ and let $H'\in M(d,\lambda n)$ be
  the $\lambda$-dilation of $H$. Let $X$ be the set of entries in $H$
  with nonzero $\Delta$-values. For $1\le i\le d$, let $A_i$ be the
  projection of $X$ onto the $i$-th coordinate.  Let $U$ be any set of
  entries of $H$ that all $(\Z_n,d)$-suitable diagonals of $H$ must
  intersect. If $\sum_i|A_i|\le(d-1)n$ then every
  $(\Z_{\lambda n},d)$-suitable diagonal of $H'$ intersects $\Psi(U)$.
\end{lemma}

\begin{proof}
  Let $D\subseteq X$ be any partial diagonal of $H$. It suffices to
  show that $D$ satisfies the hypothesis of \lref{l:dil}(ii).
  We will do this by constructing a table $T_E$, similar to the one in
  \eref{e:atran}, which lists the coordinates of the entries of a diagonal
  $E$ of $H$. We start by inscribing the coordinates of the entries in $D$
  in the first $|D|$ rows of $T_E$. Note that since $D\subseteq X$, each
  $T_E[i,j]$ inscribed so far comes from $A_j$.

  Next, for each $|D|<i\le n$ we choose some $j$ and some
  $x\in\Z_n\setminus A_j$ and set $T_E[i,j]=x$. We ensure that the choice
  of $(j,x)$ is different for different rows. The number of available choices
  for $(j,x)$ is $\sum_k(n-|A_k|)$, which is at least $n$ provided
  $\sum_i|A_i|\le(d-1)n$.

  Finally, in each column of $T_E$ we inscribe any unused elements of
  $\Z_n$ in an arbitrary order, ensuring that the rows of $T_E$ record
  the coordinates for a diagonal $E$ of $H$. By construction, the
  first $|D|$ rows of $T_E$ coordinatise the entries in $D$, while the
  remaining rows coordinatise entries outside of $X$. Hence $E\cap X=D$,
  as required.
\end{proof}

As an example application of \lref{l:dilrect}, consider the Latin
square $L_m$ whose $\Delta$ values are given in \eref{e:ord6}. Since
the nonzero $\Delta$ values are confined within a $3\times4$
submatrix, \lref{l:dilrect} will be directly applicable whenever $m>1$
(\lref{l:dil} applies when $m=1$). We see immediately that dilating
$L_m$ by any $\lambda>1$ will produce another Latin square in which
every transversal must hit one of the two entries that are the image
under $\Psi$ of entries in row $m$ of $L_m$ with nonzero
$\Delta$-value. Similar observations hold about many previously
constructed Latin squares with restricted transversals, such as the
families $\LA_n$, $\LB_n$ and $\LU_n$ constructed in \cite{EW12}.  We
can also get hypercubes of higher dimensions with strong restrictions
on their transversals by dilating hypercubes that were constructed in
\sref{s:lift}.

\section{Concluding remarks}\label{s:conc}

In \sref{s:lift} and \sref{s:dilation}, we have given methods for
constructing hypercubes with larger dimension and larger order
respectively, from smaller hypercubes. In doing so we were able to
understand the relationship between diagonals in the larger hypercubes
and those in the smaller hypercubes. This has enabled us to construct
Latin hypercubes of even order and even dimension that have relatively
few disjoint transversals (\tref{t:crampedT}) and have many entries
that are not in any transversal (\tref{t:transfree}). Unfortunately,
\cyref{cy:decomt} appears to be a fundamental obstacle to using our
methods to shed light on \cjref{cj:oddity}.  Indeed, it may be that
something much stronger than \cjref{cj:oddity} is true for $d>2$. In
this case, we do not know of any confirmed bachelor when $n$ is odd or
$d$ is odd.  And as reported in \sref{s:ord4}, it seems very common
for there to be a decomposition into transversals.

Other open questions remain. For $n\in\{4,6,8\}$ and $d>2$ it is
unclear whether \tref{t:crampedT} can be strengthened to give a result
analogous to that achieved for $n\ge10$. It may also be that
\tref{t:transfree} can be substantially strengthened. In \cite{GW25}
it is shown that Latin squares with transversals can have
asymptotically more than half of their entries not in transversals. It
would be interesting to know whether Latin hypercubes of higher
dimensions with transversals can have asymptotically more than some
fixed proportion of their entries not in any transversals.  In
\tref{t:transfree} we were only able to show that our confirmed
bachelors of order 4 have an exponentially small proportion of their
entries not in any transversal.

 
  \let\oldthebibliography=\thebibliography
  \let\endoldthebibliography=\endthebibliography
  \renewenvironment{thebibliography}[1]{%
    \begin{oldthebibliography}{#1}%
      \setlength{\parskip}{0.2ex}%
      \setlength{\itemsep}{0.2ex}%
  }%
  {%
    \end{oldthebibliography}%
  }


\begin{thebibliography}{99}

\bibitem{CW20}
B.~Child and  I.\,M.~Wanless, 
Multidimensional permanents of polystochastic matrices,
{\it Linear Algebra Appl.\/}  {\bf586} (2020), 89--102.
  
\bibitem{EW12}
J.~Egan and I.~M. Wanless,
Latin squares with restricted transversals,
{\em J. Combin. Des.} {\bf20} (2012), 124--141.


\bibitem{GW25}
A.~Ghafari and I.~M. Wanless,
Latin squares whose transversals share many entries,
arXiv:2412.12466.

\bibitem{HallJrTheorem}
M.~Hall Jr,
A combinatorial problem on abelian groups,
{\em Proc. Amer. Math. Soc.} {\bf3} (1952) 584--587.

\bibitem{MW08}
B.\,D.~McKay and I.\,M.~Wanless, 
A census of small Latin hypercubes, 
{\it SIAM J.\ Discrete Math.\/} {\bf22}, (2008) 719--736.

\bibitem{Mon24}
R. Montgomery,  Transversals in Latin squares,
London Math. Soc. Lecture Note Ser. {\bf493},
(2024) 131--158.

\bibitem{PPV24}
A.\,L.~Perezhogin, V.\,N.~Potapov, S.\,Yu.~Vladimirov,
Every Latin hypercube of order 5 has transversals,
{\em J. Combin Des.} {\bf32} (2024) 679--699.

\bibitem{Tar16}
A.\,A. Taranenko,
Permanents of multidimensional matrices: Properties and applications,
{\em J. Appl. Ind. Math.} {\bf10} (2016) 567--604.

\bibitem{Tar17}
A.\,A. Taranenko,
On the number of transversals in $n$-ary quasigroups of order 4,
\emph{Math. Notes} {\bf101} (2017), 919--921.

\bibitem{Tar18a}
A.\,A. Taranenko,
Transversals in completely reducible multiary quasigroups and in
multiary quasigroups of order 4,
Discrete Math. {\bf341} (2018), 405--420.

\bibitem{Tar18b}
A.\,A. Taranenko,
Transversals, plexes, and multiplexes in iterated quasigroups,
\emph{Electron. J. Combin.} {\bf25} (2018), \#P4.30, 17 pp.

\bibitem{Tar20}
A.\,A. Taranenko,
Positiveness of the permanent of 4-dimensional polystochastic matrices of order 4,
\emph{Discrete Appl. Math.} {\bf276} (2020), 161--165.

\bibitem{Tar21}
A.\,A. Taranenko,
Transversals, near transversals, and diagonals in iterated groups and quasigroups,
\emph{Electron. J. Combin.} {\bf28} (2021), \#P3.48, 22 pp.

\bibitem{transurv}
I.\,M.~Wanless,
``Transversals in Latin squares: A survey'', in
R.~Chapman (ed.), {\it Surveys in Combinatorics 2011},
London Math. Soc. Lecture Note Series {\bf392}, 
Cambridge University Press, 2011, pp403--437.

\bibitem{WW06}
I.\,M.~Wanless and B.\,S.~Webb, 
The existence of Latin squares without orthogonal mates,
{\it Des.\ Codes Cryptogr.} {\bf40} (2006), 131--135.

\end{thebibliography}
\end{document}